\documentclass{amsart}
\usepackage{enumitem}
\usepackage[english]{babel}
\usepackage{leftidx}
\usepackage[numbers]{natbib}
\usepackage{bigints}
\usepackage{empheq}
\usepackage{array}
\usepackage{graphicx}
\usepackage{lipsum}
\usepackage{amssymb}
\usepackage{amsthm}
\usepackage{xcolor}
\usepackage{needspace}
\usepackage{verbatim}
\usepackage[makeroom]{cancel}
\usepackage{graphicx}
\usepackage{amsmath}
\DeclareMathOperator{\sech}{sech}
\DeclareMathOperator{\csch}{csch}

\newtheorem{theorem}{Theorem}[section]
\newtheorem{proposition}[theorem]{Proposition}
\newtheorem{lema}[theorem]{Lemma}
\newtheorem{question-non}[]{}

\newtheorem{cor}[theorem]{Corollary}
\newtheorem{observation}[theorem]{Remark}
\newtheorem{example}[theorem]{Example}

\title{Gradient Einstein-Type Structures immersed into a Riemannian Warped Product}

\author{Batista, E. $^{{1},\ast}$}
\address{$^{1}$ Universidade Federal de Goi\'as, IME, 131, 74001-970, Goi\^ania, GO, Brazil.}
\email{elismardb@gmail.com $^{1}$}

\author{Adriano, L. $^{2}$}
\address{$^{2}$ Universidade Federal de Goi\'as, IME, 131, 74001-970, Goi\^ania, GO, Brazil.}
\email{levi@ufg.br $^{2}$}

\author{Tokura, W. $^{3}$}
\address{$^{3}$ Instituto Federal Goiano, 75380-000, Av. Wilton Monteiro da Rocha, s/n, Trindade, GO, Brazil.}
\email{williamisaotokura@hotmail.com $^{3}$}

\thanks{$^{\ast}$ 
Corresponding author}

\usepackage{geometry}
\begin{document}

\begin{abstract}
In this paper, we study gradient Einstein-type structure immersed into a Riemannian warped product manifold. We obtain some triviality results for the potential function and smooth map $u$. We investigate conditions for an Einstein-type structure to be totally umbilical, or totally geodesic immersed into a warped product $I\times_f M^n$. Furthermore, we study rotational hypersurface into $\mathbb{R}\times_f\mathbb{R}^n$ has a gradient Einstein-type structure. 
\end{abstract}

\maketitle

\section{Introduction}
Let $(\Sigma^n,g)$ be a connected, complete, possibly compact, Riemannian manifold
equipped with a metric of the form
\begin{equation}\label{eq0001}
\begin{cases}
    Ric_{g}^{u}+Hess h - \mu dh\otimes dh=\lambda g\\
    \tau_{g} u=du(\nabla h),
\end{cases}    
\end{equation}
where
\begin{equation*}
 Ric_{g}^u := Ric_{g}-\alpha u^{*}\langle, \rangle_{N}.   
\end{equation*}
We say that the manifold $(\Sigma^n,g)$ 
has a gradient Einstein-type structure if there are a real number $\alpha>0,  u :(\Sigma^n,g)\longrightarrow (N^p,g')$ a smooth map with tension field $\tau(u)=trace \nabla(du)$ and smooth functions $h,\mu,\lambda \in C^{\infty}(M)$, which satisfy \eqref{eq0001}. We call $h$ by the potential function and $\lambda$ the soliton function.

 To simplify the notation, we will refer to this structure as $(\Sigma^n,g,h,u,\mu,\lambda)$ and it is classified into three types according to the sign of
$\lambda$: expanding if $\lambda < 0$, steady if $\lambda = 0$ and shrinking if $\lambda > 0$.
 
 The metric defined in \eqref{eq0001} generalizes many other that have been intensively
studied by researchers in the last years. Indeed, if $h$ is a constant function, \eqref{eq0001} characterizes a harmonic Einstein metric in $\Sigma^n$, 
in this case, we also say that the metric is trivial, note that, in the particular case in which $h, u$ are constant, \eqref{eq0001} becomes an Einstein metric. In the case $\mu=0, \lambda \in \mathbb{R}$ and $u$ a constant function we obtain a gradient Ricci soliton, that are special self-similar solutions to the Ricci flow defined by R. Hamilton in \cite{MR664497}. The Ricci soliton played an extremely important role in the study of the Ricci flow, namely  the Poincaré conjecture, see \cite{perelman2002}. For more details about gradient Ricci solitons, see \cite{barbosa2014gradient, MR2243675, MR3367063, MR2112631, MR1249376, pigola2011remarks}. If $u$ is a constant map, $\lambda \in \mathbb{R}$  and $\mu=1/m,$ then \eqref{eq0001} describes a quasi-Einstein manifold for any $m>0$. Taking nonconstant functions $\mu$ and $\lambda$, we obtain the generalized $\mu-$quasi-Einstein manifold. 
For $\mu=0$ and $\lambda \in \mathbb{R}$ the metric \eqref{eq0001} defines a gradient Ricci harmonic soliton which  is a special solution of the Ricci-harmonic
flow introduced by R. Muller \cite{MR2961788}. 
The Ricci harmonic flow is the combination of harmonic map heat flow defined by Eells and Sampson \cite{MR164306} and the Ricci Flow. Thus, the study of geometric and analytical properties of these solitons and their generalizations are important for understanding the behavior of the Ricci flow and the harmonic Ricci flow. Regarding the geometry of an Einstein-type structure, it is possible to study it from two different points of view, namely, the intrinsic geometry and the submanifolds geometry. From the intrinsic point of view, Anselli \textit{et al.} in  \cite{rigoli2019} obtained several results for the Einstein-type structure introducing the concept of the $u-$curvatures. In this direction, performing with a conformal deformation of a harmonic Einstein metric, the authors obtain a solution of \eqref{eq0001}, for $n\geq 3$ and $\mu=-1/(n-2)$. Moreover, they provided basic formulas for compact Einstein-type structure in order to provide ``gap'' and rigidity results.

In relation to the study of the geometry of the submanifolds, there are many interesting and important works developed in the last decades on Ricci soliton, almost Ricci soliton and Yamabe soliton, see \cite{MR3098047, MR3367063, MR3853131}. For instance, Barros \textit{et al.} in \cite{MR3098047} studied conditions and obstructions for an isometric immersion of an almost Ricci soliton into a space form to be minimal, totally umbilical, or totally geodesic. In addition, they obtained some rigidity results for compact immersion.
On the other hand, Chen, Bang-Yen, and Deshmukh obtained a classification of Ricci solitons with concurrent potential vector field. Also, they provide a necessary and sufficient condition for a isometric immersion to be a Ricci soliton into a manifold equipped with a concurrent vector field, see \cite{MR3367063}. Abdênago Barros \textit{et al.} proved in \cite{abdenago} that a compact almost Ricci soliton with Ricci tensor being Codazzi, is isometric to the Euclidean sphere, having the height function as the potential. Aquino \textit{et al.} in \cite{MR3646891}, were studied the gradient almost Ricci solitons immersed into the space of constant sectional curvature $M^{n+1} \subset \mathbb{R}^{n+2}$, where the potential function is given by the height function from the soliton associated to a fixed direction on $\mathbb{R}^{n+2}.$
Therefore, the works above provide us with excellent motivation to study Einstein-type structures where the potential function is given by the height function. Moreover, to extend the previous works to a larger class of ambient spaces, it appears convenient to consider the immersions into a sufficiently large family of manifolds, including spaces of constant sectional curvature. A natural metric, which includes the spaces of constant sectional curvature in its range, is described by warped product metrics \cite{o1983semi}. Warped product metrics have already proven themselves to be a profitable ambient space to obtain a wide range of distinct geometrical proprieties for immersions (cf. \cite{alias2007constant, caminha2009complete, colares2012some, de2019characterizations, MR2785730}).

The aim of this paper is to study Einstein-type structures immersed into warped product space $\overline{M}^{n+1}=I\times_f M^n$, with potential function given by the height function $h=\pi_{I}\circ \phi$, where $\phi$ is an isometric immersion $\phi: \Sigma^n\longrightarrow \overline{M}^{n+1}$, $I \subset \mathbb{R}$, $M^n$ be a Riemannian manifold and $f:I\longrightarrow(0,\infty)$ a smooth function. Firstly, We investigate conditions for an Einstein-type structure immersed into a warped product to be minimal, totally umbilical, or totally geodesic. 
Secondly, we provide triviality results for the potential function $h$ and the smooth map $u$. Finally, we characterize  the rotational gradient Einstein-type hypersurfaces into $\mathbb{R}\times_f\mathbb{R}^n$.

\section{Preliminaries}


Let $\overline{M}^{n+1}= I \times M^{n}$ be Riemannian product  where $I\subset\mathbb{R}$ and $M^{n}$ be a connected, $n$-dimensional oriented Riemannian manifold. Consider on $\overline{M}^{n+1}$ the metric
\[\langle \ ,\ \rangle=\pi_{I}^{\ast}(dt^2)+f^2(\pi_{I})\pi_{M}^{\ast}(g_{M}),\] 
where $\pi_{I}$ and $\pi_{M}$ to be the canonical projection in $I$ and $M$, respectively and $f : I \rightarrow (0,\infty)$ a smooth function. This  space is called a warped product manifold with base $I$, fiber $M^n$ and warping function $f$.
In this setting, for a fixed $t_{0}\in\mathbb{R}$, we say that $\Sigma_{t_0}^n:=\{t_0\}\times M^n$ is a slice of $\overline{M}^{n+1}$.


Let $\overline{\nabla}$ and $\nabla$ the Levi-Civita connection in $I\times_{f}M^{n}$ and $\Sigma^n$, respectively. Then, the Gauss-Weingarten formulas for a isometric immersion $\phi: \Sigma^n \longrightarrow \overline{M}^{n+1}=I\times_fM^n$ are give by
\begin{equation}\label{eq1}\overline{\nabla}_{X}Y=\nabla_{X}Y+\langle AX,Y\rangle N,\qquad AX=-\overline{\nabla}_{X}N,
\end{equation}
for all $X, Y \in \mathfrak{X}(\Sigma)$, where $A:T\Sigma^n\rightarrow T\Sigma^n$ denotes the \textit{shape} operator of $\Sigma^n$ with respect to Gauss map $N$.

If we consider the height function $h:= \pi_{I}\circ \phi$ and the angle function $\theta=\langle N,\partial_{t}\rangle$ where $\partial_{t}$ is the standard unit vector field tangent to $I$, then by a straightforward computation we obtain that
\[\overline{\nabla}\pi_{I}=\langle \overline{\nabla}\pi_{I}, \partial_{t}\rangle \partial_{t}=\partial_{t},\]
so, the gradient of $h$ on $\Sigma^{n}$ is
\begin{equation}\label{eq2}\nabla h=(\overline{\nabla}\pi_{I})^{\top}=\partial_{t}^{\top}=\partial_{t}-\theta N,
\end{equation}
where $(\hspace{0,1cm}\cdot\hspace{0,1cm})^{\top}$ denotes the tangential component of a vector field
in $\mathfrak{X}(\overline{M})$ along $\Sigma^{n}$. By \eqref{eq2}, we derive
\[|\nabla h|^{2}=1-\theta^2,\]
where $|\cdot|$ denotes the norm of a vector field on $\Sigma^{n}$. 

It is well known that the Gauss equation of the immersion $\phi:\Sigma^{n}\rightarrow \overline{M}^{n+1}$ is given by
\begin{equation}\label{eq3}
\langle R(X,Y)Z,W\rangle=\langle( \overline{R}(X,Y)Z)^{\top}, W\rangle+ \langle AX,Z\rangle \langle AY,W\rangle- \langle AY, Z\rangle \langle AX, W\rangle,
\end{equation}
for every tangent vector fields $X,Y,Z$ and $W$ $\in \mathfrak{X}(\Sigma).$


Denote by $Ric$ the Ricci tensor of $\Sigma^n$ and consider a local orthonormal frame $\{E_{i}\}_{i=1}^n$ of $\mathfrak{X}(\Sigma)$. Then, it follows from the Gauss equation \eqref{eq3} that
\begin{equation}\label{Ric}
Ric^{u}(X,Y)=\sum_{i=1}^{n}\langle\overline{R}(X,E_{i})Y, E_{i}\rangle+nH\langle AX,Y\rangle-\langle AX,AY\rangle-\alpha du\otimes du(X,Y), 
\end{equation}
and
\begin{equation}\label{scal}
R^u=\sum_{i,j=1}^{n}\left<\overline{R}(E_i,E_j)E_i,E_j\right>+ n^2H^2-|A|^2-\alpha |\nabla u|^2,
\end{equation}
where $Ric^{u}=Ric-\alpha du \otimes du$ and $R^u=R-\alpha|\nabla u|^2$ is the $u-$scalar curvature of $\Sigma^n.$ When $\overline{M}^{n+1}$
is a space form of constant sectional curvature $c$, we have the identity
\begin{equation}\label{scalconst}
R^u=n(n-1)c+ n^2H^2-|A|^2-\alpha |\nabla u|^2 .
\end{equation}
Moreover, taking into account the properties of the Riemannian curvature tensor $\overline{R}$ of a warped product (see for instance Proposition 7.42 in \cite{o1983semi}), we arrive at 
\begin{align*}
\overline{R}(X,Y)Z=R^{M}(X^{\ast}, Y^{\ast})Z^{\ast}&-[(\log f)'(h)]^{2}\left[\langle X,Z\rangle Y-\langle Y,Z\rangle X\right]\\
&+(\log f)''(h)\langle Z,\partial_{t}\rangle\left[\langle Y,\partial_{t}\rangle X-\langle X,\partial_{t}\rangle Y\right]\\
&-(\log f)''(h)\left[\langle Y,\partial_{t}\rangle\langle X,Z\rangle-\langle X,\partial_{t}\rangle\langle Y,Z\rangle\right]\partial_{t},
\end{align*}
where $R^{M}$ is the curvature tensor of the fiber $M^n$ and $X^{\ast}=X-\langle X,\partial_{t}\rangle \partial_{t}$, $E_{i}^{\ast}=E_{i}-\langle E_{i},\partial_{t}\rangle \partial_{t}$ are, respectively, the projections of the tangent vector fields $X$ and $E_{i}$ onto $M^{n}$. Thus, we obtain that
\begin{equation}\label{ricc2}
\begin{split}
\sum_{i=1}^{n}\langle\overline{R}(X,E_{i})X,E_{i}\rangle=&f(h)^{-2}\sum_{i=1}^{n}K^{M}(X^{\ast}, E_{i}^{\ast})\Big{[}|X|^{2}-\langle X,\nabla h\rangle ^{2}-|X|^{2}\langle \nabla h, E_{i}\rangle ^{2}\\&-\langle X,E_{i}\rangle ^{2}+ 2\langle X, \nabla h\rangle\langle X, E_{i}\rangle\langle\nabla h, E_{i}\rangle\Big{]}+[(\log f)'(h)]^{2}\Big{(}|\nabla h|^{2}\\
&-(n-1)\Big{)}|X|^{2}-(n-2)(\log f)''(h)\langle X,\nabla h\rangle ^{2}-\frac{f''}{f}|\nabla h|^{2}|X|^{2},
\end{split}
\end{equation}
where $K^{M}$ is the sectional curvature of $M^{n}$, and hence from \eqref{scal}, the $u-$scalar curvature of $\Sigma^n$ into $I\times_fM^n$ takes the following form
\begin{equation}\label{eq4}
\begin{split}
    R^u=f(h)^{-2}&\sum_{i,j=1}^{n}K^{M}(E_{j}^{\ast}, E_{i}^{\ast})\Big{[}1-\langle E_{j},\nabla h\rangle ^{2}-\langle \nabla h, E_{i}\rangle ^{2}-\langle E_{j},E_{i}\rangle ^{2}\\
&+ 2\langle E_{j}, \nabla h\rangle\langle E_{j}, E_{i}\rangle\langle\nabla h, E_{i}\rangle\Big{]}+n[(\log f)'(h)]^{2}\left(|\nabla h|^{2}-(n-1)\right)\\
&-(n-2)(\log f)''(h)|\nabla h|^{2}-n\frac{f''}{f}|\nabla h|^{2}+n^{2}H^{2}-|A|^{2}-\alpha|\nabla u|^2.
\end{split}
\end{equation}

In what follows, we provide necessary and sufficient conditions for a hypersurface $\Sigma^n$ into $I\times_{f} M^n$ has a gradient Einstein-type structure with potential function $h=  \pi_{I}\circ \phi$.

\begin{proposition}\label{prop1}
Let $\phi: \Sigma^n \longrightarrow \overline{M}^{n+1}=I\times_{f} M^n$ be  an  isometric  immersion.   Then $(\Sigma^n,g)$ has a gradient Einstein-type structure with potential $h= \pi_{I}\circ \phi$ if, and only if,
\begin{equation}\label{fundamentaleq}
\begin{split}
    Ric^{u}(X,Y)&=\left(\lambda-\frac{f'(h)}{f(h)}\right)g(X,Y)+\left(\mu+\frac{f'(h)}{f(h)}\right)dh\otimes dh(X,Y) - \theta g(A(X),Y),\\
\tau_{g}u&=du(\nabla h).
\end{split}
\end{equation}
\end{proposition}
\begin{proof} Since $h=\pi_I\circ\phi$, we have from (\cite{MR3445380}, pg 55) that 
\begin{equation}\label{hessfuncheig}
Hessh(X,Y)=\frac{f'(h)}{f(h)}\left[g(X,Y)-dh\otimes dh(X,Y)\right]+\theta g(AX,Y).
\end{equation}
Replacing \eqref{hessfuncheig} in \eqref{eq0001} we arrive at \eqref{fundamentaleq}.
\end{proof}

The following proposition due to Caminha
\textit{et al.} \cite{MR2718145} is a generalization of the H. Hopf’s Theorem on a complete
noncompact oriented Riemannian manifold, and it will be used to prove one of our main results.
\begin{proposition}\label{proptriviality}{\textup{(\cite{MR2718145})}}
Let $X$ be a smooth vector field on the $n$ dimensional complete,
noncompact, oriented Riemannian manifold $\Sigma^n$, such that $divX$ does not change
sign on $\Sigma$. If $|X|\in L^1(\Sigma),$ then $divX = 0$ on $\Sigma$.
\end{proposition}

Next, we present to the following formula quoted from
\cite{rigoli2019}, in which it is an adaptation of the Bochner’s formula to Einstein-type structure, and will be useful in our next result.
\begin{proposition}\label{propbochner}{\textup{(\cite{rigoli2019})}}
Let $\Sigma^n$  be an $n$-dimensional Riemannian manifold with an Einstein-type structure with $\mu\in\mathbb{R}$. Then
\begin{equation}\label{bocknerets}
  \frac{1}{2}\Delta_{h}|\nabla h|^2=|Hess h|^2+\alpha|\tau_gu|^2+(2\mu\lambda n-\lambda-2\mu R^{u})|\nabla h|^2+\mu(2\mu-1)\nabla h|^4-(n-2)\langle \nabla \lambda, \nabla h\rangle,
\end{equation}
where $\Delta_{h}|\nabla h|^2=\Delta h-\langle \nabla h, \nabla |\nabla h|^2\rangle.$
\end{proposition}

\section{Examples}
Proceeding, we present examples of hypersurface with Einstein type structure isometrically immersed into $I\times_{f}M^n$.
\begin{example}
Let $(\mathbb{S}^{n},g)$ be a standard sphere immersed into Euclidean space $((0,+\infty)\times\mathbb{S}^n,g_{0})$, where $g_0=dt^2+t^2g_1$, $g_1$ is standard metric of the n-sphere. To consider $u=Id:\mathbb{S}^{n}\longrightarrow \mathbb{S}^{n}$,  if we taken the height function from the sphere, then $(\mathbb{S}^{n},g)$ have  a gradient Einstein-type structure with height function as the potential and soliton function given by $\lambda=n+1 -\alpha$.
\end{example}

\begin{example}
Consider $\Sigma^n=(0,+\infty)\times \mathbb{S}^{n-1}$ immersed into $\mathbb{R}\times \mathbb{R}^n$ furnished with metric tensor 
\begin{equation*}
 g=ds^2+\left(2 \arctan\left(\tanh\left(\frac{s}{2}\right)\right)\right)^2dv^2,   
\end{equation*}
and angle function $\theta(s)=\sqrt{1-\tanh(s)^2}$.  Hence, if we get a real number $\alpha>0$ and functions $h,u,\mu, \lambda$   given by

\begin{equation*}
\begin{split}
 h(s,v_1,...,v_{n-1})=\log(\cosh(s)),\qquad  u(s,v_1,...,v_{n-1})= \sinh(s), 
     \end{split} 
\end{equation*}

\begin{equation*}
\begin{split}
\mu(s,v_1,...,v_{n-1})&=\dfrac{\left[2\arctan\left(\tanh\left(\dfrac{s}{2}\right)\right) \csch(s)-1\right] \left[n-2+2 \arctan\left(\tanh\left(\dfrac{s}{2}\right)\right) \csch(s)\right]}{4 \arctan\left[\tanh\left(\dfrac{s}{2}\right)\right]^2} \\
&- 
 \alpha \cosh(s)^2 \coth(s)^2
\end{split}
\end{equation*}
 
 \begin{equation*}
 \begin{split}
\lambda(s,v_1,...,v_{n-1})=\frac{\sech(s)\left[4\arctan\left(\tanh(\dfrac{x}{2})\right) + (n-2) \sinh(x)\right] \tanh(x)}{4 \arctan(\tanh\left(\frac{s}{2}\right))^2},
\end{split}
\end{equation*}
we deduce that $\Sigma^n$ has a gradient Einstein-type structure with height function as the potential function (see Figure \ref{fig1} ), more details can be seen in the section \ref{sec5}.
\end{example}

\begin{figure}[!h]
\centering
\includegraphics[width=6cm]{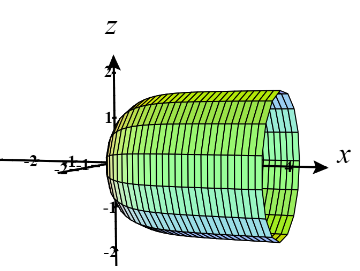}
\caption{Rotational gradient Einstein-type structure immersed into Euclidean space.}
\label{fig1}
\end{figure}

\begin{example}
Consider $\Sigma^n$ be a rotational hypersurface in $\mathbb{R}\times \mathbb{R}^n$ with constant angle function. By \cite{do2012rotation} $\Sigma^n$ can be expressed
by the Riemannian product $(0,+\infty)\times\mathbb{S}^{n-1}$ furnished with metric tensor $g = ds^2+\sigma(s)^2dv^2$, where $\sigma(s)=\theta s+c_1$. Hence, if we get the functions $h,u,\mu$ and $\lambda$ given by

\begin{equation*}
\begin{split}
 h(s,v_1,...,v_{n-1})=\sqrt{1-\theta^2}s+c_3, \qquad u(s,v_1,...,v_{n-1})= \frac{c_1e^{\sqrt{1-\theta^2}s}}{\sqrt{1-\theta^2}}+c_2
 \end{split},
 \end{equation*}
 \begin{equation*}
\begin{split}
\lambda(s,v_1,...,v_{n-1})=\frac{(\theta s+c_4)\theta\sqrt{1-\theta^2}+(n-2)(1-\theta^2)}{(\theta s+c_4)^2}, 
 \end{split}
 \end{equation*}
 \begin{equation*}
\begin{split}
\mu(s,v_1,...,v_{n-1})=\frac{-(n-2)}{(\theta s+c_4)^2}-\frac{\theta\sqrt{1-\theta^2}}{(1-\theta^2)(\theta s+c_4)}-\alpha \frac{c_1^2e^{2 \sqrt{1-\theta^2}s}}{1-\theta^2},
\end{split}
\end{equation*}
where $\alpha >0$ and  $c_1,c_2,c_3,c_4 \in \mathbb{R}$
we deduce that $\Sigma^n$ has a gradient Einstein-type structure with height function as the potential function.
\end{example}

\begin{example}
Consider $\Sigma^n=(0,\pi)\times \mathbb{S}^{n-1}$ immersed into $\mathbb{R}\times \mathbb{R}^n$ furnished with metric tensor $g=ds^2+(s/2 + \sin(2 s)/4)^2dv^2$, and angle function $\theta(s)=\cos(s)^2$.  Hence, if we get real number $\alpha>0$ and functions $h,u,\mu, \lambda$   given by

\begin{equation*}
\begin{split}
 h(s,v_1,...,v_{n-1})=-\frac{\sqrt{1-\cos ^4(s)} \left(\left(\sqrt{\cos (2 s)+3}\right) \cot (s)+\sqrt{2} \csc (s) \log \left(\sqrt{2}
\cos (s)+\sqrt{\cos (2 s)+3}\right)\right)}{2 \sqrt{\cos (2 s)+3}}, 
     \end{split} 
\end{equation*}
\vspace{0.5cm}
\begin{equation*}    
\begin{split}      
 \mu(s,v_1,...,v_{n-1})&=-\dfrac{1}{(1-\cos(s)^4)}\left(\dfrac{2\cos(s)^2\sqrt{1-\cos(s)^4}}{s+\cos(s)\sin(s)}-2 \cos(s) \sin(s)\left(\dfrac{\cos(s)^2}{\sqrt{1-\cos(s)^4}}\right.\right.\\
 &+\left.\left.\dfrac{2(n-1)}{s + \cos(s)\sin(s)}\right)+\dfrac{
  16(n-2-\alpha c_1^2-(n-2)\cos(s)^4)
  +8 (s+\cos(s)\sin(s))\sin(2s)}{(2s + \sin(2s))^2}\right),
 \end{split}
 \end{equation*}
 \vspace{0.5cm}
 \begin{equation*}
 \begin{split}
\lambda(s,v_1,...,v_{n-1})&=\dfrac{16(n-2-\alpha c_1^2-(n-2)\cos(s)^4)+8(s+\cos(s) \sin(s))\sin(2s)}{(2s + \sin(2s))^2}\\
&+\dfrac{2\cos(s)^2 \sqrt{1-\cos(s)^4})}{(s +\cos(s)\sin(s))},
\end{split}
\end{equation*}

 \begin{equation*}
 \begin{split}
 u(s,v_1,...,v_{n-1})=c_1v_k+c_2, \qquad \mbox{for some}\quad  k \in \{1,2,...,n-1\} \quad \mbox{and}\quad c_1,c_2 \in \mathbb{R}.
\end{split}
\end{equation*}
We deduce that $\Sigma^n$ has a gradient Einstein-type structure with height function as the potential function.
\end{example}

\begin{figure}[!h]
\centering
\includegraphics[width=6cm]{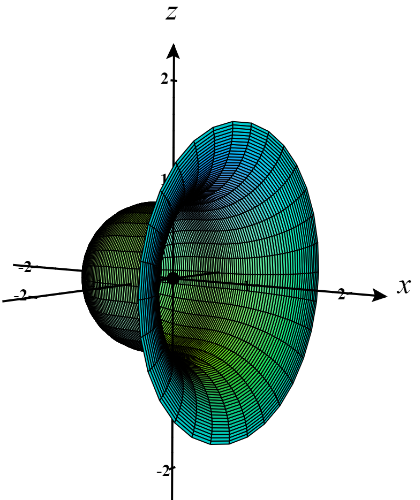}
\caption{Rotational gradient Einstein-type structure immersed into Euclidean space.}
\label{figufibra}
\end{figure}

\section{Results and Proofs}
In \cite{rigoli2019} the authors proved that any compact Einstein-type manifold with bounded Ricci tensor sa-\\
tisfying some inequalities involving the functions $\mu$ and $\lambda$ is Einstein harmonic. In this way, our first result provide a necessary condition for a compact immersion with a gradient Einstein-type structure to be trivial.

\begin{theorem}\label{theocompwp1}
Let $(\Sigma^{n},g,h,u,\mu,\lambda)$ be a compact Einstein-type structure immersed on $I\times_{f}M^n$ with $h=\pi_{I}\circ \phi$. If the mean curvature of $\Sigma^n$ satisfies $0\leq H\leq (\log f)'(h)$, then $\Sigma^n$ is Einstein harmonic.
\end{theorem}

\begin{proof}
From the trace in \eqref{hessfuncheig},
we deduce
\begin{equation*}
\Delta h=\frac{f'(h)}{f(h)}\left(n-|\nabla h|^2\right)+n\theta H.
\end{equation*}
Therefore, 
\begin{equation}\label{Delta}
\begin{split}
\Delta h+\langle \nabla(\log f(h)),\nabla h\rangle&=n\frac{f'(h)}{f(h)}+n\theta H=n\left(\frac{f'(h)}{f(h)}+\theta H\right)\geq0.
\end{split}
\end{equation}
Hence, it follows from the maximum principle, see page 35 of \cite{MR3445380}, that $h$ is constant, so $\Sigma^n$ is a slice. Finally, taking into account equation \eqref{eq0001}, we get that $\Sigma^n$ is an Einstein harmonic manifold.
\end{proof}

\begin{observation}
Note that a direct consequence of the Theorem \ref{theocompwp1} is that whenever $\Sigma^n$ is compact, we have that $M^n$ is compact.
\end{observation}

In the particular case in which the ambient space is a space form, we obtain the
following classifications.
\begin{cor}Let $(\Sigma^{n},g,h,u,\mu,\lambda)$ be a compact Einstein-type structure immersed on Euclidean space $(0,+\infty)\times_{t}\mathbb{S}^n$ with $h=\pi_{I}\circ \phi$. If the mean curvature of $\Sigma^n$ satisfies $0\leq H\leq h^{-1}$, then $\Sigma^n$ is isometric to a round sphere.
\end{cor}

\begin{proof}From Theorem \ref{theocompwp1}, we have that $\Sigma^n=\{t_{0}\}\times \mathbb{S}^n$, with $h(x)=t_{0}\in (0,+\infty)$, and from the induced metric, we get that $g=t_{0}^{2}g_{\mathbb{S}^{n}}$, i.e., $(\Sigma^n,g)$ is isometric to Euclidean sphere $(\mathbb{S}^n, t_{0}^{2}g_{\mathbb{S}^n})$.
\end{proof}

\begin{cor}Let $(\Sigma^{n},g,h,u,\mu,\lambda)$ be a compact Einstein-type structure immersed on Euclidean sphere $(0,\pi)\times_{\sin t}\mathbb{S}^n,$ with $h=\pi_{I}\circ \phi$. If the mean curvature of $\Sigma^n$ satisfies $0\leq H\leq\cot{h}$, then $\Sigma^n$ is isometric to a round sphere.
\end{cor}

\begin{cor}Let $(\Sigma^{n},g,h,u,\mu,\lambda)$ be a compact Einstein-type structure immersed on Hyperbolic sphere $(0,+\infty)\times_{\sinh t}\mathbb{S}^n,$ with $h=\pi_{I}\circ \phi$. If the mean curvature of $\Sigma^n$ satisfies $0\leq H\leq\coth{h}$, then $\Sigma^n$ is isometric to a round sphere.
\end{cor}

In \cite{MR3098047} the authors considered  isometric immersion of an almost Ricci soliton $(M^n, g, X, \lambda)$ in spaces form $\overline{M}(c)^{n+p}$ of sectional curvature $c$. They proved that if $|X| \in L^1(M)$ and $\lambda \geq (n-1)(c+H^2)$, then $M^n$ is totally umbilical submanifold in $\overline{M}(c)^{n+p}$.  Our next result provides a necessary condition for a hypersurface with an Einstein-type structure immersed in $I\times_{f}M^n$ to be totally geodesic or totally umbilical. In addition, we derive conditions for nonexistence of minimal immersion of an gradient Einstein-type structure into a warped product.

\begin{theorem}
Let $\phi:\Sigma^n\longrightarrow I\times_fM^n$ be an isometric immersion of a non-compact Einstein-type structure $(\Sigma^{n},g,h,u,\mu,\lambda)$ whose fiber $M^n$ has sectional curvature $k_M\leq \inf_{I}(f'^2-ff'')$ with $\mu\in \mathbb{R},\mu\geq 1/2$ and $\Sigma^{n}$ not being a slice, then the following statements hold 

\begin{itemize}

 \item [\textup{(a)}] 
If $|e^{-h}\nabla |\nabla h|^2| \in L^{1}(\Sigma), \langle\nabla\lambda, \nabla h \rangle\leq 0$ and the soliton function satisfies
\begin{equation*}
 \lambda\geq \frac{-2\mu}{2\mu n-1}\left(n(n-1)\frac{f''(h)}{f(h)}-n^2H^2+\alpha |\nabla u|^2\right),  
\end{equation*}
then, $\Sigma^n$ is totally geodesic hypersurface of $I\times_fM^n$.

\item [\textup{(b)}] If $|e^{-h}\nabla |\nabla h|^2| \in L^{1}(\Sigma), \langle\nabla\lambda, \nabla h \rangle\leq 0$ and the soliton function satisfies
\begin{equation*}
  \lambda\geq \frac{-2\mu}{2\mu n-1}\left(n(n-1)\left(\frac{f''(h)}{f(h)}-H^2\right)+\alpha |\nabla u|^2\right),  
\end{equation*}
then, $\Sigma^n$ is totally umbilical hypersurface of $I\times_fM^n$.
\end{itemize}
\end{theorem}

\begin{proof} By Proposition \ref{propbochner} we have 

\begin{equation}\label{00003}
\begin{split}
  e^{h}Div(e^{-h}\nabla |\nabla h|^2)=
  \frac{1}{2}\Delta_{h}|\nabla h|^2\geq
  (2\mu\lambda n-\lambda-2\mu R^{u})|\nabla h|^2
  +\mu(2\mu-1)|\nabla h|^4\\-(n-2)\langle \nabla \lambda, \nabla h\rangle.
  \end{split}
\end{equation}
Using the hypothesis for estimate the $u-$scalar curvature of $\Sigma$ in \eqref{eq4}, we obtained 

\begin{equation}\label{eq4444}
\begin{split}
    R^u&\leq \frac{\inf_{I}(f'^2-ff'')}{f^2}(n-1)\big{(}n-2|\nabla h|^2\big{)}+n[(\log f)'(h)]^{2}\left(|\nabla h|^{2}-(n-1)\right)\\
&\quad-(n-2)(\log f)''(h)|\nabla h|^{2}-n\frac{f''}{f}|\nabla h|^{2}+n^{2}H^{2}-|A|^{2}\\
&\leq-(n-1)(\log f)''(h)\big{(}n-2|\nabla h|^2\big{)}+n[(\log f)'(h)]^{2}\left(|\nabla h|^{2}-(n-1)\right)\\
&\quad-(n-2)(\log f)''(h)|\nabla h|^{2}-n\frac{f''}{f}|\nabla h|^{2}+n^{2}H^{2}-|A|^{2}\\
&\leq -n(n-1)\frac{f''(h)}{f(h)}+n^{2}H^{2}-|A|^{2}-\alpha|\nabla u|^2.
\end{split}
\end{equation}

Replacing \eqref{eq4444} in \eqref{00003} we deduce  

\begin{equation}\label{00006}
\begin{split}
  e^{h}Div(e^{-h}\nabla |\nabla h|^2)&\geq
  (2\mu\lambda n-\lambda-2\mu R^{u})|\nabla h|^2
  +\mu(2\mu-1)|\nabla h|^4-(n-2)\langle \nabla \lambda, \nabla h\rangle\\
  &\geq \left(2\mu\lambda n-\lambda+2\mu n(n-1)\frac{f''(h)}{f(h)}-2\mu n^2H^2+2\mu\alpha |\nabla u|^2\right )|\nabla h|^2\\
  &+2\mu|A|^2|\nabla h|^2+\mu(2\mu-1)|\nabla h|^4-(n-2)\langle \nabla \lambda, \nabla h\rangle\geq 0.
  \end{split}
\end{equation}
Therefore, the item a) follows from the Proposition \ref{proptriviality}. Now, equation \eqref{00006} implies that
\begin{equation*}
\begin{split}
  e^{h}Div(e^{-h}\nabla |\nabla h|^2)&\geq \left(2\mu\lambda n-\lambda+2\mu n(n-1)\frac{f''(h)}{f(h)}-2\mu n(n-1)H^2+2\mu\alpha |\nabla u|^2\right)|\nabla h|^2\\
  &+2\mu|\Phi|^2|\nabla h|^2+\mu(2\mu-1)|\nabla h|^4-(n-2)\langle \nabla \lambda, \nabla h\rangle\geq 0,
  \end{split}
\end{equation*}
where $\Phi$ is the traceless second fundamental form of $\Sigma^n$, namely, $\Phi=A-HI$, which satisfies the following equality $|\Phi|^{2}=tr(\Phi^2)=|A|^{2}-nH^{2}\geq0$. Applying the Proposition \ref{proptriviality} we conclude the result.
\end{proof}

\begin{theorem}
Let $\phi: \Sigma^n \longrightarrow I\times_{f} M^{n}$, to be a isometric immersion of a gradient Einstein-type structure, with poitential function $h=\pi_I\circ\phi,$ $0<\mu \in \mathbb{R}$ and $f$ a convex function. If $\lambda>0$ and $k_{M}\leq f'^2$, then $\phi$ can not be minimal.
\end{theorem}
\begin{proof}
The following inequalities, in the sense of quadratic forms, hold
\begin{equation*}
0\leq u^{*}\langle, \rangle_{N}\geq |du|^2g.
\end{equation*}
Hence, \eqref{eq0001} becomes
\begin{equation*}
Ric_{g}+Hess h - \mu dh\otimes dh\geq (\lambda+\alpha|du|^2)g\geq \lambda g.
\end{equation*}
Since $\sup |\nabla h|<\infty$, then from \cite{MR2373611} we conclude that $(\Sigma^n,g)$ is a compact gradient Einstein-type structure with finite fundamental group. On the other hand, assuming that $\Sigma^n$ is minimal and using the properties of the warped product metric we have that the sectional curvature of $\overline{M}^{n+1
}=I\times_{f} M^{n}$ is

\begin{equation*}
K_{\partial_t V}=-\frac{f''}{f}\leq 0, \qquad K_{V W}= \frac{K_{M}-f'^2}{f^2}\leq 0,
\end{equation*}
where $V,W \in \mathfrak{X}(M).$ Thus, by \cite{MR187183} we have that $(\Sigma^n,g)$ is a minimal compact gradient Einstein-type structure with infinite fundamental group. So, we obtain a contradiction and this completes the proof of the theorem.

\end{proof}

In sequence,  we provide a necessary condition for a compact gradient Einstein-type structure immersed into a Riemannian warped product be totally umbilical hypersurface, where we do not consider hypotheses about the soliton  function $\lambda$ and sectional curvature of $M^n.$ 

\begin{theorem}\label{theototalumbilcomp}
Let $(\Sigma^{n},g,h,u,\mu,\lambda)$ to be a compact manifold with gradient Einstein-type structure immersed into a Riemannian warped product $I\times_{f} M^{n}$. If  $R^{u}$ is constant, $\mu=-f'(h)/f(h)$ and $\theta$ does not change sign, then $\Sigma^{n}$ is a totally umbilical hypersurface of $I\times_{f} M$.
\end{theorem}
\begin{proof}
From fundamental equation \eqref{fundamentaleq} and Theorem 5.21 in \cite{rigoli2019} we obtained 
\begin{equation*}
    \begin{split}
        \xi g(X,Y)-\left(\lambda-\frac{f'(h)}{f(h)}\right)g(X,Y)=-\theta g(X,Y),
    \end{split}
\end{equation*}
for some function $\xi \in C^{\infty}(\Sigma),$ thus, for a local orthonormal $\{E_{i}\}_{i=1}^n$ of $\mathfrak{X}(\Sigma)$ associated with the Weingarten operator, i.e., $A(E_{i})=\lambda_{i}E_{i}$, where $\{\lambda_{i}\}_{i=1}^{n}$ are the principal curvatures of $\Sigma^n$, we have
\begin{equation*}
    \begin{split}
        \left(\xi-\left(\lambda-\frac{f'(h)}{f(h)}\right)\right)\delta_{ij}=-\theta\lambda_{i} \delta_{ij},
    \end{split}
\end{equation*}
which implies that 
\begin{equation*}\lambda_{i}=-\theta^{-1}\left(\xi-\left(\lambda-\frac{f'(h)}{f(h)}\right)\right),\qquad 1\leq i\leq n.
\end{equation*}
Therefore, $\Sigma^n$ is totally umbilical with mean curvature $H=-\theta^{-1}\left(\xi-\left(\lambda-\frac{f'(h)}{f(h)}\right)\right).$
\end{proof}

We pointed out that an interesting question is known in what conditions a gradient Einstein-type structure has $u-$constant map. In such a case the Einstein-type structure turns out a generalized quasi-Einstein, see for instance \cite{MR4061465, MR3378388, MR3229621}. In the following, by means of a inequality involving the soliton function, we provide a condition for an Einstein-type structure to has $u-$constant map.

\begin{theorem}\label{theotrivialityu}Let $(\Sigma^{n},g,h,u,\mu,\lambda)$ be a gradient Einstein-type structure immersed into $I\times_{f}M^n$, where fiber $M^n$ has sectional curvature $k_M\leq \inf_{I}(f'^2-ff'').$ If 
\begin{equation*}
    \lambda \geq \frac{f'(h)}{f(h)}-(n-1)\frac{f''(h)}{f(h)}+\left|\mu + \frac{f'(h)}{f(h)}\right|+H(nH+\theta),
\end{equation*}
then $u$ is a constant map.
\end{theorem}

\begin{proof}
From the tracing in fundamental equation \eqref{fundamentaleq} of Proposition \ref{prop1}  we have
\begin{equation}
\begin{split}\label{eq333}
    R^{u}&=n\left(\lambda-\frac{f'(h)}{f(h)}\right)+\left(\mu+\frac{f'(h)}{f(h)}\right)|\nabla h|^2 - n\theta H.
\end{split}
\end{equation}
Comparing \eqref{eq4444} and \eqref{eq333} we get
\begin{equation*}
   n\left(\lambda-\frac{f'(h)}{f(h)}\right)+\left(\mu+\frac{f'(h)}{f(h)}\right)|\nabla h|^2 - n\theta H   \leq-n(n-1)\frac{f''(h)}{f(h)}+n^{2}H^{2}-|A|^{2}-\alpha|\nabla u|^2.
\end{equation*}
Therefore, since $|\nabla h|^2\leq 1,$ we have
\begin{equation*}
\begin{split}
    \alpha|\nabla u|^2 \leq-n(n-1)\frac{f''(h)}{f(h)}+n^2H^2-nH^2-|\phi|^2-n\left(\lambda-\frac{f'(h)}{f(h)}\right)\\-\left(\mu+\frac{f'(h)}{f(h)}\right)|\nabla h|^2+n\theta H\\\leq n\left(-(n-1)\frac{f''(h)}{f(h)}-\lambda+\frac{f'(h)}{f(h)}+\left|\mu+\frac{f'(h)}{f(h)}\right|+H(nH+\theta)\right)\leq 0.
    \end{split}
\end{equation*}
Thus, $u$ is a constant map.
\end{proof}

Next, we obtain a characterization for minimal gradient Einstein-type structure immersed into a space form. The proof is an immediate consequence of the Theorem \ref{theotrivialityu}.
\begin{cor}
Any minimal gradient Einstein-type structure $(\Sigma^{n},g,h,u,\mu,\lambda)$ immersed on $\mathbb{R}^{n+1}$ with $\lambda\geq|\mu|$ is a $\mu$-quasi Einstein structure, i.e., $u$ is a constant map.
\end{cor}
\begin{cor}
Any minimal gradient Einstein-type structure $(\Sigma^{n},g,h,u,\mu,\lambda)$ immersed on hyperbolic space $\mathbb{R}\times_{e^t}\mathbb{R}^n$ with $\lambda\geq|\mu+1| -(n-2)$ is a $\mu$-quasi Einstein structure, i.e., $u$ is a constant map.
\end{cor}
\begin{cor}
Any minimal gradient Einstein-type structure $(\Sigma^{n},g,h,u,\mu,\lambda)$ immersed on Euclidean sphere $(0,\pi)\times_{\sin t}\mathbb{S}^n$ with 
$$\lambda\geq\cot(h)+(n-1)|\mu+\cot(h)|,$$
is a $\mu$-quasi Einstein structure, i.e., $u$ is a constant map.
\end{cor}

Note that in Theorem \ref{theocompwp1}, we assume the compacteness of $\Sigma$ and we get a trivialization for the potential function $h$. Now, assuming $\Sigma$ non-compact, we deduce the following similar result.

\begin{theorem}\label{theoharmoniceinsteinwp}
Let $(\Sigma^{n},g,h,u,\mu,\lambda)$ be an non-compact gradient Einstein-type structure with $\mu>-\frac{f'}{f}$ immersed into $I\times_{f}M^n$ whose fiber $M^n$ has sectional curvature $k_M\leq \inf_{I}(f'^2-ff'').$ If $A(\nabla h)=\rho\nabla h$ for some smooth function $\rho$, then the following  statements holds
\begin{itemize}
\item [\textup{(a)}] If $\Delta:=(nH+\theta)^2+4\Bigg{[}\dfrac{f'}{f}-(n-1)\dfrac{f''}{f}-\lambda\Bigg{]}\leq0$, then $(\Sigma^n,g)$ is Einstein harmonic.
    \item [\textup{(b)}] If $\Delta:=(nH+\theta)^2+4\Bigg{[}\dfrac{f'}{f}-(n-1)\dfrac{f''}{f}-\lambda\Bigg{]}>0$ and 
    \[\rho\leq\frac{nH+\theta-\sqrt{\Delta}}{2},\qquad\text{or}\qquad\frac{nH+\theta+\sqrt{\Delta}}{2}\leq \rho,\]
 then $(\Sigma^n,g)$ is Einstein harmonic.
\end{itemize}
\end{theorem}

\begin{proof}From the curvature assumption $k_{M}\leq \inf_{I}(f'^2-ff'')$ jointly with equations \eqref{Ric} and \eqref{ricc2} we arrive at
\begin{equation*}
\begin{split}
Ric_{\Sigma}(\nabla h, \nabla h)&=\sum_{i=1}^{n}\langle\overline{R}(\nabla h,E_{i})\nabla h,E_{i}\rangle+nH\langle A(\nabla h),\nabla h\rangle-\langle A(\nabla h),A(\nabla h)\rangle\\ &\quad\leq
-\log f''(h)\sum_{i=1}^{n}\Big{[}|\nabla h|^{2}-|\nabla h|^{4}-|\nabla h|^{2}\langle \nabla h, E_{i}\rangle ^{2}-\langle \nabla h,E_{i}\rangle ^{2}+ 2|\nabla h|^{2}\langle \nabla h, E_{i}\rangle^{2}\Big{]}\\&\quad+[(\log f)'(h)]^{2}\Big{(}|\nabla h|^{2}-(n-1)\Big{)}|\nabla h|^{2}-(n-2)(\log f)''(h)|\nabla h|^{4}-\frac{f''(h)}{f(h)}|\nabla h|^{4}\\
&\quad+nH\langle A(\nabla h),\nabla h\rangle-\langle A(\nabla h),A(\nabla h)\rangle\\
&\leq -\frac{f''(h)}{f(h)}(n-1)|\nabla h|^2+nH\langle A(\nabla h),\nabla h\rangle-\langle A(\nabla h),A(\nabla h)\rangle.
\end{split}
\end{equation*}
By the fundamental equation \eqref{fundamentaleq}, we get

\begin{equation*}
\begin{split}
    Ric_{\Sigma}^{u}(\nabla h,\nabla h)&=\left(\lambda-\frac{f'(h)}{f(h)}\right)|\nabla h|^2+\left(\mu+\frac{f'(h)}{f(h)}\right)|\nabla h|^4 - \theta \langle A(\nabla h),\nabla h\rangle,
\end{split}
\end{equation*}
and taking into account that $Ric_{\Sigma}(\nabla h,\nabla h)=Ric_{\Sigma}^u(\nabla h,\nabla h)+\alpha|\tau_g u|^2$ we deduce

\begin{equation*}
\begin{split}
0\leq\alpha|\tau_g u|^2+\left(\mu+\frac{f'}{f}\right)|\nabla h|^4\leq\left(\frac{f'}{f}-\frac{f''}{f}(n-1)-\lambda+(nH+\theta)\rho-\rho^2\right)|\nabla h|^2.
\end{split}
\end{equation*}
The result follows from the analysis of the roots of

\begin{equation*}
-\rho^2+(nH+\theta)\rho+\frac{f'}{f}-(n-1)\frac{f''}{f}-\lambda=0.
\end{equation*}
\end{proof}
In the same settings as the above theorem, if the angle function $\theta$ is constant, then, we obtain that $\nabla h$ is an eigenvector of $A$ with $A(\nabla h)=-\cos(\theta)\frac{f'}{f}\nabla h$, see \cite{MR2785730}. So, we obtained straight of Theorem \ref{theoharmoniceinsteinwp}  the following characterization.

\begin{cor}
Let $(\Sigma^{n},g,h,u,\mu,\lambda)$ be an gradient Einstein-type structure with $\mu>-\frac{f'}{f}$ immersed into $I\times_{f}M^n$ whose fiber $M^n$ has sectional curvature $k_M\leq \inf_{I}(f'^2-ff'')$. If $\Sigma^n$ has constant angle and the items a) and b) of the Theorem \ref{theoharmoniceinsteinwp} hold, then $\Sigma^n$ is Einstein Harmonic.
\end{cor}

\begin{cor}Let $(\Sigma^{n},g,h,u,\mu,\lambda)$ be an gradient Einstein-type structure with $\mu>-cot(t)$ immersed on Euclidean sphere $(0,\pi)\times_{\sin t}\mathbb{S}^n$. If $\Sigma^n$ has constant angle and the soliton function satisfies
\begin{equation*}
   \lambda\geq \cot(t)+n-1,
\end{equation*}
then $\Sigma^n$ is Einstein harmonic.
\end{cor}

\begin{cor}
Let $(\Sigma^{n},g,h,u,\mu,\lambda)$ be an gradient Einstein-type structure with $\mu>-1$ immersed on Hyperbolic space $\mathbb{R}\times_{e^t}\mathbb{R}^n$. If $\Sigma^n$ has constant angle and the soliton function satisfies
\begin{equation*}
  \lambda\geq 2-n,
\end{equation*}
then $\Sigma^n$ is Einstein harmonic.
\end{cor}

\begin{cor}
Let $(\Sigma^{n},g,h,u,\mu,\lambda)$ be an gradient Einstein-type structure with $\mu>0$ immersed on Euclidean space $\mathbb{R}\times\mathbb{R}^n$. If $\Sigma^n$ has constant angle and the soliton function $\lambda\geq 0$
then $\Sigma^n$ is Einstein harmonic.
\end{cor}

In the previous results we considered isometric  immersion of gradient Einsten-type structure where the smooth map $u:\Sigma^n\longrightarrow N^p$ was not fix. If we choose $u=\phi$, we obtain the following results considering an isometric immersion of an Einstein-type gradient structure into the warped product $I\times_fM^n$.

\begin{theorem}
Let $\phi:\Sigma^n\longrightarrow I\times_fM^n$ be a isometric immersion of an gradient Einstein-type structure $(\Sigma^{n},g,h,u,\mu,\lambda)$ whose fiber $M^n$ has sectional curvature $k_M\leq \inf_{I}(f'^2-ff'')$, with $\phi=u$, then the following statements holds 

\begin{itemize}

 \item [\textup{(a)}] 
There is no isometric immersion of a gradient Einstein-type structure into $I\times_fM^n$ such that $\lambda\geq -(n-1)\frac{f''(h)}{f(h)}$.

\item [\textup{(b)}] If $\lambda\geq -(n-1)\frac{f''(h)}{f(h)}-\frac{\alpha}{n}|\nabla u|^2$, then $\Sigma^n$ is Einstein harmonic totally geodesic in $I\times_fM^n$. Furthermore, the smooth map $u$ is given by $|\nabla u|^2=n\left(-(n-1)\frac{f''(h)}{f(h)}-\lambda\right).$
\end{itemize}
\end{theorem}

\begin{proof}
First, if $u$ is an isometric immersion, then $\tau_gu = nH$ (see \cite{MR3445380}). Note that $du (\nabla h)$ is tangent to $\Sigma,$ while $nH$ is normal to $\Sigma$, therefore, by \eqref{eq0001} it follows that $\Sigma^n$ is harmonic Einstein minimally immersed in $I\times_f M^n$. In addition, by \cite{rigoli2019} we get that $\lambda$ is constant.
Replacing $R^u=n\lambda$ in \eqref{scal}, we obtain

\begin{equation*}
    \alpha|\nabla u|^2\leq -n(n-1)\frac{f''(h)}{f(h)}-n\lambda-|A|^2\leq 0,
\end{equation*}
therefore, $u$ is a constant map, which is a contradiction, since $u$ is a isometric immersion. Now, regarding item (b), assuming the hypothesis in sectional curvature of $M^n$, we have

\begin{equation*}
    |A|^2\leq n\left(-(n-1)\frac{f''(h)}{f(h)}-\lambda-\frac{\alpha}{n}|\nabla u|^2\right)\leq 0,
\end{equation*}
thus, $\Sigma^n$ is totally geodesic with $|\nabla u|^2=n\left(-(n-1)\frac{f''(h)}{f(h)}-\lambda\right).$
\end{proof}

\begin{cor}
There is no isometric immersion $\phi$ of a gradient Einstein-type structure shrinking into $\mathbb{S}^{n+1}$ with $u=\phi$.
\end{cor}

\begin{cor}
Let $(\Sigma^{n},g,h,u,\mu,\lambda)$ be a compact gradient Einstein-type structure immersed by $\phi$ into a space form $\overline{M}(c)^{n+1}$ of sectional curvature $c$. If $\lambda\geq (n-1)c-\frac{\alpha}{n}|\nabla u|^2$ and $u=\phi$ , then $\Sigma^n$ is isometric to $\mathbb{S}^n$.
\end{cor}

\section{Rotational hypersurface with Einstein type structure}\label{sec5}
\label{classification}


In this section, we present a characterization of \textit{rotational hypersurface} gradient Einstein-type immersed into $\mathbb{R}\times_{f}\mathbb{R}^{n}$ with potential $h:=\pi_{I}\circ \phi$ and angle function $|\theta|< 1$. Following Dajczer and do Carmo paper \cite{do2012rotation},
we shall use the terminology of rotational hypersurface in $\mathbb{R}\times_{f}\mathbb{R}^{n}$ as a hypersurface invariant by the orthogonal group $O(n)$ seen as a subgroup of the isometries group of $\mathbb{R}\times_{f}\mathbb{R}^{n}$.



Initially, consider the coordinates $(t,x_{1},\dots,x_{n})$, as well as the standard orthonormal basis $\{\eta_{1}, \dots , \eta_{n+1}\}$ of $\mathbb{R}\times_{f}\mathbb{R}^{n}$. Then, up to isometry, we can assume the rotation axis to be $\eta_{1}$. Consider a parametrized by the arc length curve in the $tx_{n}$ plane given by
\begin{align*}
      \gamma \colon (t_{0}&,t_{1}) \longrightarrow \mathbb{R}\times_{f}\mathbb{R}^n\\
       &s \xmapsto{\hspace{0.5cm}}(\zeta(s),0,\dots,0,\beta(s)).
\end{align*}
Rotating this curve around the $t$-axis we obtain a \textit{rotational hypersurface} in $\mathbb{R}\times_{f}\mathbb{R}^{n}$. Now, in order to obtain a parametrization of a rotational hypersurface, consider the stander unit sphere given by $\mathbb{S}^{n-1}\subset\mathbb{R}^{n}=\textup{span}\{\eta_{2},\dots, \eta_{n+1}\}$ with orthogonal parametrization given by
\begin{align*}
&X_{1}=\cos v_{1},\quad X_{2}=\sin v_{1}\cos v_{2},\quad X_{3}=\sin v_{1}\sin v_{2}\cos v_{3}, \quad \dots\quad\\ 
&X_{n-1}=\sin v_{1}\sin v_{2}.\dots\sin v_{n-2}\cos v_{n-1}, X_{n}=\sin v_{1}\sin v_{2}\dots\sin v_{n-2}\sin v_{n-1}.
\end{align*}
Therefore, a parametrization of a rotational hypersurface $\Sigma^n$ with radial axis $\eta_{1}$ into $\mathbb{R}\times_{f}\mathbb{R}^n$ is given by
\begin{equation}\label{para}
\begin{split}
  \phi \colon &(t_{0},t_{1})\times (0,2\pi)^{n-1} \rightarrow \mathbb{R}\times_{f}\mathbb{R}^n\\[1ex]
  &(s,v_{1},\dots, v_{n-1}) \xmapsto{\hspace{0.3cm}} \zeta(s)\eta_{1}+\beta(s)X(v_{1},\dots,v_{n-1}),
  \end{split}
\end{equation}
where $$X(v_{1},\dots,v_{n-1})=(0,X_{1}(v_{1},\dots,v_{n-1}),\dots, X_{n}(v_{1},\dots,v_{n-1})).$$

In this setting, we provide the following results.

\begin{theorem}\label{teorotacionallyfull}
Let $\phi:\Sigma^{n}\rightarrow\mathbb{R}\times_{f}\mathbb{R}^{n}$ be a rotational hypersurface with  angle function $|\theta|< 1$. Then, $\Sigma^{n}$ have a gradient Einstein type structure, if, only if the system of equation below is satisfied.  
\begin{equation}\label{eqrotfulllambda}
\begin{split}
\lambda&= -(n-1)(1-\theta^2)\dfrac{f''(h)}{f(h)}
 -(n-1)\theta\frac{\sqrt{1-\theta^2}}{\sigma}(\log f)'(h)\\
 &+\left[(n-1)\left(\theta[(\log f)'(h)]-\frac{\sqrt{1-\theta^2}}{\sigma}\right)-\theta\right]\frac{\theta '}{\sqrt{1-\theta^2}}-\alpha\frac{|\tau_g u|^2}{1-\theta^2}-\mu(1-\theta^2),
\end{split}
 \end{equation}
 
 \begin{equation}\label{eq223}
 \begin{split}
 \lambda-\frac{f'(h)}{f(h)}-\theta\left(\frac{\sqrt{1-\theta^2}}{\sigma}-\frac{f'(h)}{f(h)}\theta\right)=(n-2)\left(\frac{\sqrt{1-\theta^2}}{\sigma}-\frac{f'(h)}{f(h)}\theta\right)^2-(1-\theta^2)(\log f)''\\
 +(\theta^2-(n-1))((\log f)')^2-\frac{f'(h)}{f(h)}\theta\frac{\sqrt{1-\theta^2}}{\sigma} -\frac{\theta'}{\sigma}+ \frac{\theta'}{\sqrt{1-\theta^2}}\frac{f'(h)}{f(h)}\theta -\alpha \frac{du(\phi_{v_i})^2}{\sigma^2},
 \end{split}
  \end{equation}
  
  \begin{equation}\label{eq222}
\begin{split}
du(\phi_{s}) du(\phi_{v_{i}})=0,
 \end{split}
\end{equation}

  \begin{equation}\label{eqharmonicrotfibra}
\begin{split}
du(\phi_{v_{i}}) du(\phi_{v_{j}})=0, \qquad \forall i\neq j,
 \end{split}
\end{equation}
  
 \begin{equation}\label{eqharm}
 \tau_g u= \sqrt{1-\theta^2}du(\phi_s).
 \end{equation}
\end{theorem}


\begin{proof}
Since $\phi:\Sigma^{n}\rightarrow\mathbb{R}\times_{f}\mathbb{R}^{n}$ is a rotational hypersurface, we deduce from \eqref{para} that 
\begin{equation}\label{tangente}
\begin{split}
    \phi_{s}&=\zeta'(s)\eta_{1}+\beta'(s)X,\\[1.0ex] \phi_{v_{i}}&=\beta(s)X_{v_{i}},\quad 1\leq i\leq n-1,
    \end{split}
\end{equation}
and then, the first fundamental form of $\Sigma^n$ takes the form


\begin{align}\label{I}
  I &=\begin{bmatrix}
    \quad 1 & 0 & \dots & 0\\[0.6em]
    \quad 0 & f(\zeta(s))^{2}\beta(s)^2 & \dots & 0 \\[0.6em]
    \quad \vdots & \vdots & \ddots & \vdots \\[0.6em]
    \quad 0 & 0 & \dots & f(\zeta(s))^{2}\beta(s)^{2}
  \end{bmatrix}.
\end{align}
From equation \eqref{I} we notice that the induced metric on $\Sigma^n$ can be expressed by the warped product metric $g=ds^2+\sigma(s)^{2}dv^2$ where $\sigma(s)=f(\zeta(s))\beta(s)$. In which case, it follows from the Levi-Civita connection on the warped product metric that:
\begin{equation}\label{xxx}
\begin{split}
&\nabla_{\phi_{s}}\phi_{s}=0,\\[1.5ex] &\nabla_{\phi_{s}}\phi_{v_{i}}=\nabla_{\phi_{v_{i}}}\phi_{s}=\frac{\sigma_{s}}{\sigma}\phi_{v_{i}},\\[1.5ex] &\nabla_{\phi_{v_{i}}}\phi_{v_{j}}=\phi_{v_{i} v_{j}}-\sigma\sigma_{s}\delta_{ij}\phi_{s}.
\end{split}
\end{equation}

From the tangent components \eqref{tangente}, we easily derive the following unit normal vector field for $\Sigma^n$
\begin{equation*}
    N=f(\zeta(s))\beta'(s)\eta_{1}-\frac{\zeta'(s)}{f(\zeta(s))}X(v_{1},\dots,v_{n-1}).
\end{equation*}
Hence,
\begin{equation}\label{unit2}\theta(s)=\langle \partial_{t},N\rangle=f(\zeta(s))\beta'(s).
\end{equation}
Since $\gamma(s)=(\zeta(s),0,\dots,0,\beta(s))$ is parametrized by the arc length curve, i.e.,
\begin{equation*}
\zeta'(s)^{2}+f(\zeta(s))^{2}\beta'(s)^{2}=1.
\end{equation*}
From \eqref{unit2}, we deduce $\zeta'(s)=\sqrt{1-\theta(s)^2}$, whose general solution is given by 
\begin{equation}\label{a}
\zeta(s)=\int\sqrt{1-\theta^2}ds.
\end{equation}
Then, replacing equation \eqref{a} into \eqref{unit2} and solving in $s$, we derive the following expression
\begin{equation}\label{b}
    \beta(s)=\int^{s}\frac{\theta(w)}{f(\zeta(w))}dw.
\end{equation}
Therefore, the rotational hypersurface takes the following form
\begin{equation}\label{sa}
\begin{split}
\phi=\left(\int^s\sqrt{1-\theta^2}ds\right)\eta_{1}+\left(\int^{s}\frac{\theta}{f(\zeta(w))}dw\right)X(v_{1},\dots,v_{n-1}).
\end{split}
\end{equation}

Now, in order to compute the Weingarten operator $A_{N}$, let us consider the following decomposition
\begin{equation}\label{deriva}\partial_{t}=\sqrt{1-\theta^2}\phi_{s}+\theta N.
\end{equation}
Taking the covariant derivative of \eqref{deriva} with respect $\phi_{v_{i}}$, as well as the properties of the Levi-Civita connection of $\mathbb{R}\times_{f}\mathbb{R}^{n}$ (Proposition 7.35 in \cite{o1983semi}), we arrive at
\begin{equation}\label{xx}
\nabla_{\phi_{v_{i}}}\phi_{s}=\frac{\theta}{\sqrt{1-\theta^2}} A_{N}\phi_{v_{i}}+\frac{1}{\sqrt{1-\theta^2}}\frac{f'(\zeta(s))}{f(\zeta(s))}\phi_{v_{i}},\qquad \forall i\in\{1,\dots, n-1\},
\end{equation}
where we use the trivial result $\overline{\nabla}_X\partial_t=\frac{f'}{f}\left(X-\langle X, \partial_t\rangle\partial_t\right)$ for all $X \in \mathfrak{X}(\Sigma).$ Combining \eqref{xxx} and \eqref{xx}, yields

\begin{equation}\label{sigma}
    \sqrt{1-\theta^2}\frac{\sigma_{s}}{\sigma}\phi_{v_{i}}=\theta A_{N}\phi_{v_{i}}+\frac{f'(\zeta(s))}{f(\zeta(s))}\phi_{v_{i}},
\end{equation}
and therefore, from the expression of $\sigma$, we obtain that $\phi_{v_{i}}$ is an eigenvector for $A_{N}$ and satisfies

\begin{equation}\label{weing2}
   A_{N}\phi_{v_{i}}= \left(\frac{\sqrt{1-\theta^2}}{\sigma}-\frac{f'(\zeta(s))}{f(\zeta(s))}\theta\right)\phi_{v_{i}}.
\end{equation}
On the other hand, taking the covariant derivative of \eqref{deriva} with respect $X\in\mathfrak{X}(\Sigma)$ and using the Gauss-Weingarten formulas \eqref{eq1}, we deduce the following implications
\begin{equation*}
    \begin{split}
        \overline{\nabla}_{X}\partial_{t}&=X(\sqrt{1-\theta^2})\phi_s+\sqrt{1-\theta^2}\overline{\nabla}_{X}\phi_{s}+X(\theta)N+ \theta\overline{\nabla}_{X}N\\
        &=X(\sqrt{1-\theta^2})\phi_s+\sqrt{1-\theta^2}\nabla_{X}\phi_{s}+\sqrt{1-\theta^2}g(A_{N}\phi_{s},X)N+X(\theta)N-\theta A_{N}X,
    \end{split}
\end{equation*}
and, again from the proprieties of the Levi-Civita connection of $\mathbb{R}\times_{f}\mathbb{R}^{n}$ \cite{o1983semi}, it follows
\begin{equation}\label{comparison}
\begin{split}
\frac{f'(\zeta(s))}{f(\zeta(s))}\left(X-\sqrt{1-\theta^2}g(X,\phi_{s})\partial_{t}\right)&=X(\sqrt{1-\theta^2})\phi_s+\sqrt{1-\theta^2}\nabla_{X}\phi_{s}+\sqrt{1-\theta^2}g(A_{N}\phi_{s},X)N \\
&+X(\theta)N -\theta A_{N}X.
\end{split}
\end{equation}
Comparing the tangent and the normal parts of \eqref{comparison}, one gets that $\phi_{u}$ is an eigenvector for $A_{N}$ and satisfies
\begin{equation}\label{weing1}
    A_{N}\phi_{s}=\left(-\frac{f'(\zeta(s))}{f(\zeta(s))}\theta-\frac{\theta'}{\sqrt{1-\theta^2}}\right) \phi_{s}.
\end{equation}
Therefore, from \eqref{weing1} and \eqref{weing2}, we conclude that $\{\phi_{s}, \phi_{v_{1}},\dots, \phi_{v_{n-1}}\}$ form an orthogonal basis of $A_{N}$ and its expression on that basis takes the form
\begin{align}\label{Wei}
  A_{N} &=\begin{bmatrix}
    -\dfrac{f'(\zeta(s))}{f(\zeta(s))}\theta-\dfrac{\theta'}{\sqrt{1-\theta^2}} & 0 & \dots & 0\\
    0 & \dfrac{\sqrt{1-\theta^2}}{\sigma}-\dfrac{f'(\zeta(s))}{f(\zeta(s))}\theta & \dots & 0 \\
    \vdots & \vdots & \ddots & \vdots \\
     0 & 0 & \dots & \dfrac{\sqrt{1-\theta^2}}{\sigma}-\dfrac{f'(\zeta(s))}{f(\zeta(s))}\theta
  \end{bmatrix}.
\end{align}

Now, since we are suppose that $(\Sigma^{n},g,h,u,\mu,\lambda)$ to have a Einstein type structure, we obtain from Proposition \ref{prop1} that 
\begin{equation}\label{sse}
    \begin{split}
    Ric^{u}(X,Y)&=\left(\lambda-\frac{f'(h)}{f(h)}\right)g(X,Y)+\left(\mu+\frac{f'(h)}{f(h)}\right)dh\otimes dh(X,Y) - \theta g(A(X),Y)
\end{split}.
\end{equation}

Notice that, in particular cases $X=\phi_{s}$, $Y=\phi_{v_{i}}$ and $X=\phi_{v_{i}}$, $Y=\phi_{v_{j}}$, $i\neq j$, the orthogonality of $X$, $Y$ and the expression for the height function 
\begin{equation}\label{height}h(s,v_{1},\dots,v_{n})=(\pi_{\mathbb{R}}\circ\phi)(s,v_{1},\dots,v_{n})=\int^s\sqrt{1-\theta^2}ds,
\end{equation}
implies that equation \eqref{sse} become 

\begin{equation}\label{eqrot1}
    Ric^{u}(\phi_{s},\phi_{v_{i}})=0,
\end{equation}
and
\begin{equation}\label{eqrot2}
    Ric^{u}(\phi_{v_{i}},\phi_{v_{j}})=0,
\end{equation}
for all $i\neq j$.
 Hence, it is necessary to check the equation \eqref{sse} for a pair of fields $X=Y=\phi_{s}$ and $X=Y=\phi_{v_i}$.

For $X=Y=\phi_{s}$, we obtain
\begin{equation}\label{eqrot06}
    \begin{split}
    Ric^{u}(\phi_{s},\phi_{s})&=\left(\lambda-\frac{f'(h)}{f(h)}\right)g(\phi_{s},\phi_{s})+\left(\mu+\frac{f'(h)}{f(h)}\right)dh\otimes dh(\phi_{s},\phi_{s}) - \theta g(A(\phi_{s}),\phi_{s})\\
    &=\lambda-\frac{f'(h)}{f(h)}+\left(\mu+\frac{f'(h)}{f(h)}\right)(1-\theta^2)-\theta g\left( -\frac{f'(\zeta(s))}{f(\zeta(s))}\theta-\frac{\theta'}{\sqrt{1-\theta^2}} \phi_{s},\phi_{s}\right)\\
    &=\lambda+\mu(1-\theta^2)+\frac{\theta \theta '}{\sqrt{1-\theta^2}}.
\end{split}
\end{equation}

Now, for $X=Y=\phi_{v_{i}}$, with $1\leq i\leq n-1$, we get
\begin{equation}\label{eqrot07}
    \begin{split}
    Ric^{u}(\phi_{v_{i}},\phi_{v_{i}})&=\left(\lambda-\frac{f'(h)}{f(h)}\right)g(\phi_{v_{i}},\phi_{v_{i}})+\left(\mu+\frac{f'(h)}{f(h)}\right)dh\otimes dh(\phi_{v_{i}},\phi_{v_{i}}) - \theta g(A(\phi_{v_{i}}),\phi_{v_{i}})\\
    &=\sigma^2\left(\lambda-\frac{f'(h)}{f(h)}-\theta\left(\frac{\sqrt{1-\theta^2}}{\sigma}-\frac{f'(h)}{f(h)}\theta\right)\right).
\end{split}
\end{equation}

On the other hand, we obtained from equation \eqref{Ric} and \eqref{ricc2} that

\begin{equation}\label{eqrot3}
\begin{split}
 Ric^{u}(\phi_s,\phi_{v_{i}})&= [(\log f)'(h)]^{2}\Big{(}|\nabla h|^{2}-(n-1)\Big{)}\langle \phi_s,\phi_{v_{i}}\rangle -(n-2)(\log f)''(h)\langle \phi_s,\nabla h\rangle \langle \phi_{v_{i}},\nabla h\rangle\\
 &-\frac{f''}{f}|\nabla h|^{2}\langle \phi_s,\phi_{v_{i}}\rangle+nH\langle A_{N}(\phi_s),\phi_{v_{i}}\rangle-\langle A_{N}(\phi_s),A_{N}\phi_{v_{i}}\rangle-\alpha du\otimes du(\phi_s,\phi_{v_{i}})\\
 &=-\alpha du(\phi_s) du(\phi_{v_{i}}), \qquad \forall \quad 1 \leq i\leq n-1.
 \end{split}
\end{equation}

Similarly, we have to
\begin{equation}\label{eqrot4}
\begin{split}
 Ric^{u}(\phi_{v_{i}},\phi_{v_{j}})=-\alpha du(\phi_{v_{i}}) du(\phi_{v_{j}}),
 \end{split}
\end{equation}
for all $i\neq j.$

Using equations \eqref{Ric} and \eqref{ricc2} again, we have to

\begin{equation}\label{eqrot05}
\begin{split}
 Ric^{u}(\phi_s,\phi_s)&= [(\log f)'(h)]^{2}\Big{(}|\nabla h|^{2}-(n-1)\Big{)}\langle \phi_s,\phi_s\rangle -(n-2)(\log f)''(h)\langle \phi_s,\nabla h\rangle \langle \phi_s,\nabla h\rangle\\
 &-\frac{f''}{f}(h)|\nabla h|^{2}\langle \phi_s,\phi_s\rangle+nH\langle A_{N}(\phi_s),\phi_s\rangle-\langle A_{N}(\phi_s),A_{N}\phi_s\rangle-\alpha du\otimes du(\phi_s,\phi_s)\\
 &=-(n-1)(1-\theta^2)(\log f )''(h)-(n-1)(1-\theta^2)\left[(\log f)'(h)\right]^2\\
 &-(n-1)\theta\frac{\sqrt{1-\theta^2}}{\sigma}[(\log f)'(h)]+(n-1)\left(\theta[(\log f)'(h)]-\frac{\sqrt{1-\theta^2}}{\sigma}\right)\frac{\theta '}{\sqrt{1-\theta^2}}-\alpha\frac{|\tau_g u|^2}{1-\theta^2},  
 \end{split}
\end{equation}

and 

\begin{equation}\label{eqrot08}
\begin{split}
\sigma^{-2} Ric^{u}(\phi_{v_{i}},\phi_{v_{i}})&=(n-2)\left(\frac{\sqrt{1-\theta^2}}{\sigma}-\frac{f'(h)}{f(h)}\theta\right)^2-(1-\theta^2)(\log f)''\\
 &+(\theta^2-(n-1))((\log f)')^2-\frac{f'(h)}{f(h)}\theta\frac{\sqrt{1-\theta^2}}{\sigma}-\frac{\theta'}{\sigma}+ \frac{\theta'}{\sqrt{1-\theta^2}}\frac{f'(h)}{f(h)}\theta-\alpha\dfrac{ du(\phi_{v_i})^2}{\sigma^2}.
\end{split}
\end{equation}

Combining the equations
\eqref{eqrot1} and \eqref{eqrot3}, we obtain \eqref{eq222}. 

From equations \eqref{eqrot2},  \eqref{eqrot4} we have \eqref{eqharmonicrotfibra}. By \eqref{eqrot06} and \eqref{eqrot05} we derive \eqref{eqrotfulllambda}. Finally, blending \eqref{eqrot07} and \eqref{eqrot08} we arrive at \eqref{eq223}. Note that the \eqref{eqharm} equation can be obtained replacing $\nabla h$ in the second equation of \eqref{fundamentaleq}, and this concludes the proof of the theorem. 
\end{proof}

The authors in \cite{batista2019warped} provide that the map smooth $u$ of Ricci harmonic soliton in a warped product $B^n\times_fF^m$ splits  like $u=u_{B}\circ\pi_{B}$ or $u=u_{F}\circ\pi_{F}$ when $h=h_B\circ\pi_B$ and $u$ is a real function, where $u_{B} \in C^{\infty}(B)$, $u_{F} \in C^{\infty}(F)$  and $h_{B} \in C^{\infty}(B).$  In this sense, we characterize rotational symmetric gradient Einstein-type structures into $I\times_{f}\mathbb{R}^n$ with angle function $|\theta|< 1$ and potential function $h=\pi_I\circ\phi$.

\begin{theorem}
Let $\phi:\Sigma^{n}\rightarrow\mathbb{R}\times_{f}\mathbb{R}^{n}$ be a rotational hypersurface with angle function $|\theta|<1$. Then, $\Sigma^{n}$ have a gradient Einstein-type structure if, only if the functions $f, h, \lambda, \mu, u$ verify:

\begin{enumerate}[label=(\roman*)]
\item If $u=u_{I}\circ \pi_{I},$ then
\begin{equation}
 \begin{split}
 \lambda&=-(1-\theta^2)(\log f)''(h)-(1-\theta^2)(n-1)(\log(f)'(h))^2\\&+\left(1-\theta^2-(2n+3)\theta\dfrac{\sqrt{1-\theta^2}}{\sigma}+\dfrac{\theta\theta '}{\sqrt{1-\theta^2}}\right)\log(f)'(h)
 + \dfrac{\theta\sqrt{1-\theta^2}-\theta'}{\sigma}+(n-2)\dfrac{(1-\theta^2)}{\sigma^2}
 \end{split}
  \end{equation}

\begin{equation}
\begin{split}
\mu&=-(n-2)(\log (f))''(h)-\left(1-(n+4)\dfrac{\theta}{\sigma\sqrt{1-\theta^2}}-(n-2)\theta\theta'(1-\theta^2)^{3/2}\right)(\log(f))'(h)\\
&-\dfrac{(\theta\sqrt{1-\theta^2}-(n-2)\theta')}{\sigma(1-\theta^2)}+\theta\theta'(1-\theta^2)^{3/2}-\alpha \dfrac{(u')^2}{1-\theta^2}-\dfrac{n-2}{\sigma^2},
\end{split} 
 \end{equation}
 \vspace{0.5cm}
  \begin{equation*}
   h = \int^s \sqrt{1-\theta^2}dw+c_1,
 \end{equation*}

 \begin{equation}\label{eqrotharmo}
   u = \int^s c_2e^{\int^s \sqrt{1-\theta^2}dw}dw+c_3.
 \end{equation}
 
 \item If $u=u_{\mathbb{R}^n}\circ \pi_{\mathbb{R}^n},$ then
 
 \begin{equation}
 \begin{split}
 \lambda&=-(1-\theta^2)(\log f)''(h)-(1-\theta^2)(n-1)(\log(f)'(h))^2\\&+\left(1-\theta^2-(2n+3)\theta\dfrac{\sqrt{1-\theta^2}}{\sigma}+\dfrac{\theta\theta '}{\sqrt{1-\theta^2}}\right)\log(f)'(h)
 \\
 &+ \dfrac{\theta\sqrt{1-\theta^2}-\theta'}{\sigma}+(n-2)\dfrac{(1-\theta^2)}{\sigma^2}-\dfrac{\alpha(u')^2}{\sigma^2},
 \end{split}
  \end{equation}

\begin{equation}
\begin{split}
\mu&= -(n-2)(\log (f))''(h)-\left(1-(n+4)\dfrac{\theta}{\sigma\sqrt{1-\theta^2}}-(n-2)\theta\theta'(1-\theta^2)^{3/2}\right)(\log(f))'(h)\\
&-\dfrac{(\theta\sqrt{1-\theta^2}-(n-2)\theta')}{\sigma(1-\theta^2)}+\theta\theta'(1-\theta^2)^{3/2}+\alpha \dfrac{(u')^2}{(1-\theta^2)\sigma^2}-\dfrac{n-2}{\sigma^2},
\end{split}
 \end{equation}
  \vspace{0.5cm}
 \begin{equation*}
   h = \int^s \sqrt{1-\theta^2}dw+c_1,
 \end{equation*}
 \begin{equation}
    u(v_k) = c_4v_k+c_5,
 \end{equation}
 where $c_1, c_2,c_3,c_4 \quad \mbox{and}\quad c_5 \in \mathbb{R}.$
 
\end{enumerate}
\end{theorem}
\begin{proof}
Initially, consider $u=u_{I}\circ\pi_{I}$. So, the equation \eqref{eqharm} of Theorem \ref{teorotacionallyfull} becomes

 \begin{equation}\label{eqedoharm}
    \frac{u''}{u'} = \sqrt{1-\theta^2}.
 \end{equation}
Integrating \eqref{eqedoharm} in relation to $s$, we obtain \eqref{eqrotharmo}. Since $du(\phi_{v_i})=0$ we get the item (i) replacing \eqref{eqrotharmo} in the \eqref{eqrotfulllambda} equation. Now, consider $u=u_{\mathbb{R}^n}\circ\pi_{\mathbb{R}^n},$ by equation \eqref{eqharmonicrotfibra} we deduce that the smooth map $u$ depends of a unique $v_k$ for some $k=\{1,...,n-1\},$ in other words $u$ is a function of $\mathbb{R}$ in $\mathbb{R}$. Therefore, by equation \eqref{eqharm} we deduce that $u(v_k)=c_3v_k+c_4$, where $c_3, c_4 \in \mathbb{R}.$ The conclusion follow directly from Theorem \ref{teorotacionallyfull}.
\end{proof}

\bibliographystyle{abbrv}

\begin{thebibliography}{9}

\bibitem{MR4061465}
M. Ahmad Mirshafeazadeh and B. Bidabad. 
On the Rigidity of Generalized Quasi-Einstein Manifolds. 
\textit{Bull. Malays. Math. Sci. Soc.}, 43(2):2029{2042, 2020.}

\bibitem{alias2007constant}
 L. J. Alías and M. Dajczer. 
Constant mean curvature hypersurfaces in warped product spaces.
\textit{Proceedings of the Edinburgh Mathematical Society,} 50(3):511{526, 2007.}

\bibitem{MR3445380}
L. J. Alías, P. Mastrolia, and M. Rigoli. 
Maximum principles and geometric applications.
\textit{Springer Monographs in Mathematics. Springer,} Cham, 2016.

\bibitem{rigoli2019}
A. Anselli, G. Colombo, and M. Rigoli. 
On the geometry of einstein-type structures.
\textit{Preprint,} 08 2019.

\bibitem{MR3646891}
C. P. Aquino, H. F. de Lima, and J. N. V. Gomes. 
Characterizations of immersed gradient almost
Ricci solitons.
\textit{Pacific J. Math.,} 288(2):289{305, 2017.}

\bibitem{barbosa2014gradient}
E. Barbosa, R. Pina, and K. Tenenblat. 
On gradient ricci solitons conformal to a pseudo-euclidean
space.
\textit{Israel Journal of Mathematics,} 200(1):213{224, 2014.}

\bibitem{abdenago}
A. Barros, J. N. Gomes, and E. Ribeiro. 
A note on rigidity of the almost ricci soliton.
 \textit{Archiv der Mathematik,} 100(5):481{490, 2013.}

\bibitem{MR3098047}
A. Barros, J. N. Gomes, and E. Ribeiro, Jr. 
Immersion of almost Ricci solitons into a Riemannian
manifold.
 \textit{Mat. Contemp.,} 40:91{102, 2011.}

\bibitem{batista2019warped}
E. Batista, L. Adriano, and W. Tokura. 
On warped product gradient ricci-harmonic soliton,
\textit{Preprint,} 06, 2019.

\bibitem{caminha2009complete}
A. Caminha, H. F. de Lima, et al. 
Complete vertical graphs with constant mean curvature in
semi-riemannian warped products.
\textit{Bulletin of the Belgian Mathematical Society-Simon Stevin,}
16(1):91{105, 2009.}

\bibitem{MR2718145}
A. Caminha, P. Souza, and F. Camargo. 
Complete foliations of space forms by hypersurfaces.
\textit{Bull. Braz. Math. Soc.} (N.S.), 41(3):339{353, 2010.}


\bibitem{MR2243675}
 H.-D. Cao. 
Geometry of Ricci solitons.
\textit{Chinese Ann. Math. Ser. B,} 27(2):121{142, 2006.}


\bibitem{MR3367063}
B.-Y. Chen and S. Deshmukh. Ricci solitons and concurrent vector fields.
 \textit{Balkan J. Geom. Appl.},20(1):14{25, 2015.}


\bibitem{MR3853131}
B.-Y. Chen and S. Deshmukh. 
Yamabe and quasi-Yamabe solitons on Euclidean submanifolds.
\textit{Mediterr. J. Math.,} 15(5):Paper No. 194, 9, 2018.


\bibitem{MR2112631}
S.-C. Chu. 
Basic properties of gradient Ricci solitons.
\textit{In Geometric evolution equations,} volume 367 of Contemp. Math., pages 79{102. Amer. Math. Soc., Providence, RI, 2005}.


\bibitem{colares2012some}
A. G. Colares and H. F. De Lima. 
Some rigidity theorems in semi-riemannian warped products.
\textit{Kodai Mathematical Journal,} 35(2):268{282, 2012.}

\bibitem{de2019characterizations}
E. L. de Lima and H. F. de Lima. Characterizations of minimal hypersurfaces immersed in certain
warped products. 
\textit{Extracta mathematicae,} 34(1):123{134, 2019.}

\bibitem{MR2785730}
F. Dillen, M. I. Munteanu, J. Van der Veken, and L. Vrancken. Classification of constant angle
surfaces in a warped product. 
\textit{Balkan Journal of Geometry and Its Applications,} 16(2):35{47, 2011.}

\bibitem{do2012rotation}
M. do Carmo and M. Dajczer. Rotation hypersurfaces in spaces of constant curvature. In 
\textit{Manfredo P. do Carmo Selected Papers,} pages 195{219. Springer, 2012.}

\bibitem{MR164306}
J. Eells, Jr. and J. H. Sampson. Harmonic mappings of Riemannian manifolds. 
\textit{Amer. J. Math.,} 86:109{160, 1964.}

\bibitem{MR187183}
T. Frankel. On the fundamental group of a compact minimal submanifold. 
\textit{Ann. of Math.} (2), 83:68{73, 1966.}

\bibitem{perelman2002}
P. Grisha. The entropy formula for the Ricci flow and its geometric applications.
\textit{arXiv:math/0211159,} 2002.

\bibitem{MR664497}
R. S. Hamilton. Three-manifolds with positive Ricci curvature. 
\textit{J. Differential Geometry,} 17(2):255{
306, 1982.}

\bibitem{MR3378388}
Z. Hu, D. Li, and J. Xu. On generalized m-quasi-Einstein manifolds with constant scalar curvature.
\textit{J. Math. Anal. Appl.,} 432(2):733{743, 2015.}

\bibitem{MR1249376}
T. Ivey. Ricci solitons on compact three-manifolds. 
\textit{Differential Geom. Appl.,} 3(4):301{307, 1993.}

\bibitem{MR2961788}
R. Müller. Ricci flow coupled with harmonic map flow. 
\textit{Ann. Sci. Éc. Norm. Supér. (4),} 45(1):101{
142, 2012.}

\bibitem{o1983semi}
B. O’neill. Semi-Riemannian geometry with applications to relativity, volume 103. 
\textit{Academic press,}
1983.

\bibitem{pigola2011remarks}
S. Pigola, M. Rimoldi, and A. G. Setti. Remarks on non-compact gradient ricci solitons. 
\textit{Mathematische Zeitschrift,} 268(3-4):777{790, 2011.}

\bibitem{MR3229621}
D. G. Prakasha and H. Venkatesha. Some results on generalized quasi-Einstein manifolds. 
\textit{Chin. J. Math. (N.Y.),} pages Art. ID 563803, 5, 2014.

\bibitem{MR2373611}
W. Wylie. Complete shrinking Ricci solitons have finite fundamental group. 
\textit{Proc. Amer. Math.} Soc., 136(5):1803{1806, 2008.}
\end{thebibliography}

\end{document}